\documentclass[11pt]{article}
\usepackage[usenames]{color}
\usepackage{amsmath,amsfonts,amsthm}
\usepackage{syntonly}
\usepackage{setspace}
\usepackage[mathcal]{euscript}
\usepackage{xypic}
\usepackage{hyperref}
\usepackage[all]{xy}
\usepackage{color}  

\def\CC{{\mathbb C}}

\def\NN{{\mathbb N}}

\def\RR{{\mathbb R}}

\def\ZZ{{\mathbb Z}}

\def\cT{{\cal T}}

\def\cV{{\cal V}}

\def\Der{\hbox{{\rm{Der}}}}
\def\Rank{\hbox{{\rm{rank}}}}

\def\wgd{\hbox{{\rm{w.gl.dim}}}}

\def\Ker{\hbox{{\rm{Ker}}}}
\def\lpd{\hbox{{\rm{l.pd}}}}
\def\rpd{\hbox{{\rm{r.pd}}}}
\def\gr{\hbox{{\rm{gr}}}}

\def\Inf{\hbox{{\rm{inf}}}}
\def\Grade{\hbox{{\rm{grade}}}}
\def\Ht{\hbox{{\rm{ht}}}}
\def\Codim{\hbox{{\rm{codim}}}}
\def\Sym{\hbox{{\rm{Sym}}}}

\def\Ext{{\hbox{\rm{Ext}}}}
\def\Ass{{\hbox{\rm {Ass}}}}
\def\Hom{{\hbox{\rm{Hom}}}}
\def\Dim{{\hbox{\rm{dim}}}}

\def\InjDim{{\hbox{\rm{inj.dim}}}}
\def\Supp{{\hbox{\rm{Supp}}}}
\def\Spec{{\hbox{\rm{Spec}}}}

\def\Ord{{\hbox{\rm{ord}}}}
\def\Ad{{\hbox{\rm{Adj}}}}

\newtheorem{Teo}{Theorem}[section]
\newtheorem{Lemma}[Teo]{Lemma}
\newtheorem{Prop}[Teo]{Proposition}
\newtheorem{Cor}[Teo]{Corollary}

\newtheorem*{MainEx}{Main Example}
\theoremstyle{definition}
\newtheorem{Def}[Teo]{Definition}
\newtheorem{Rem}[Teo]{Remark}

\newtheorem{Hyp}[Teo]{Hypothesis}
\begin{document}
\title{On certain rings of differentiable type and finiteness properties of local cohomology}
\author{Luis N\'u\~nez-Betancourt}
\maketitle
\abstract{ 
Let $R$ be a commutative $F$-algebra, where $F$ is a field of characteristic $0$, satisfying the following conditions:
$R$ is equidimensional of dimension $n$, every residual field with respect to a maximal ideal is an algebraic extension of $F,$ and
$\Der_F (R)$ is a finitely generated projective $R$-module of rank $n$ such that $R_m\otimes_R \Der_F (R)=\Der_F(R_m)$.
We show that  
the associated graded ring of the ring of differentiable operators, $D(R,F)$, is a commutative
Noetherian regular with unity and pure graded dimension equal to $2\dim(R)$.
Moreover, we prove that $D(R,F)$ has weak global dimension equal to $\dim(R)$ and that  its Bernstein class  is closed under localization at one element.
Using these properties of $D(R,F)$, we show that the set of associated primes of every local cohomology module, $H^i_I(R)$, is finite. 
If $(S,m,K)$ is a complete regular local ring of mixed characteristic $p>0$, we show that the localization of $S$ at $p$, $S_p$, 
is such a ring. As a consequence, the set of associated primes of $H^i_I (S)$ that does not 
contain $p$ is finite. Moreover, we prove this finiteness property for a larger class of functors.
}

\section{Introduction}
Let $R$ denote a commutative Noetherian ring with unity. 
If $M$ is an $R$-module and $I\subset R$ is an ideal, we denote the $i$-th 
cohomology of $M$ with support in $I$ by $H^i_I (M)$. The structure of these modules has been widely studied
by several authors. Among the results obtained, one encounters the following finiteness properties for certain 
regular rings.

\begin{itemize}
\item[\rm{(1)}] \quad The set of associated primes of $H^i_I(R)$ is finite;
\item[\rm{(2)}] \quad The Bass numbers of $H^i_I(R)$ are finite;
\item[\rm{(3)}] \quad $\InjDim H^i_I (R)\leq \Dim\Supp H^i_I(R)$.
\end{itemize}
Lyubeznik showed  ($1$), ($2$) and ($3$) hold in
the local case of equal characteristic $0$ \cite{LyuDMod}. His technique relies in the use of $D$-modules 
in a power series ring with coefficients over a field of characteristic $0$. Later, in his work on  local cohomology 
for unramified mixed characteristic $p>0$ \cite{LyuUMC}, 
Lyubeznik used rings of differentiable type to prove that the set of associated primes not containing $p$ is finite.
Our goal is to develop the theory of $D$-modules in a more general setting in order to prove a similar 
result for any regular local ring of mixed characteristic, namely: 

\begin{Teo}\label{MainOne}
Let $R$ be a regular commutative Notherian ring with unity that contains a field, $F$, of characteristic $0$ satisfying the following conditions: 
\begin{itemize}
\item[\rm{(1)}] $R$ is equidimensional of dimension $n$;
\item[\rm{(2)}] every residual field with respect to a maximal ideal is an algebraic extension of $F$;
\item[\rm{(3)}] $\Der_F (R)$ is a finitely generated projective $R$-module of rank $n$ such that $R_m\otimes_R \Der_F (R)=\Der_F(R_m)$.
\end{itemize}
Then, the ring of $F$-linear differential operators $D(R,F)$ is a ring of differentiable type of weak global dimension equal to $\Dim(R)$. 
Moreover, the Bernstein class of $D(R,F)$ is closed under localization at one element.
\end{Teo}
This theorem generalizes some of the results of Mebkhout and Narv\'aez-Macarro about certain rings of differentiable type \cite{MeNa}. 
There, $R$ is a commutative Noetherian 
regular ring that contains a field, $F$, of characteristic zero satisfying $(1)$, $(2)$, but instead of ($3$) in Theorem \ref{MainOne},
there exist $F$-linear derivations $\partial_1,\ldots, \partial_n \in \Der_{F_0} (R)$ and $a_1\ldots, a_n\in R$ such that 
$\partial_i a_j=1$ if $i=j$ and $0$ otherwise.

\begin{Teo}\label{MainTwo}
Let $(R,m,K)$ be a regular commutative Notherian local ring of mixed characteristic $p>0$. Then, the set of associated primes of $H^i_I (R)$ that do not 
contain $p$
is finite for every $i\in \NN$ and every ideal $I\subset R$.
\end{Teo}

In fact, we prove this finiteness property for a larger class of rings and functors.

This manuscript is organized as follows. In section $2$, we study certain rings of differentiable type, 
and we prove Theorem \ref{MainOne}. Later, in section $3$, 
we generalize some results about properties of the Bernstein class
of rings of differentiable type obtained by  Mebkhout and  Narv\'aez-Macarro \cite{MeNa}. Finally, in section $4$, we show the 
finiteness of the associated primes of the local cohomology
for certain rings in characteristic $0$; moreover, we prove this result for a larger class of functors. 
As a consequence, we obtain Theorem \ref{MainTwo}.

\section{Rings of differentiable type}

We start by recalling a theorem from  Matsumura's book:

\begin{Teo}[ Teo. $99$ in \cite{Matsumura} ]\label{99}
Let $(R,m ,F)$ be a regular local commutative Notherian ring with unity of dimension $n$ containing a field $F_0$. Suppose that $F$ is
an algebraic separable extension of $F_0$. Let $\hat{R}$ denote the completion of $R$ with respect to $m$.
Let $x_1,\ldots,x_n$ be a regular system of parameters of $R$. Then, $\widehat{R}=F[[x_1,\dots,x_n]]$ the power 
series ring with coefficients in $F$, and $\Der_F {\hat{R}}$ is a free $\widehat{R}$-module with
basis $\partial / \partial x_1, \ldots,\partial / \partial x_n$. Moreover, the following conditions are equivalent:
\begin{itemize}
\item $\partial / \partial x_i$ ($i=1,\ldots , n$) maps $R$ into $R$, i.e. $\partial / \partial x_i\in Der_{F_0}(R)$;
\item there exist $D_1,\ldots, D_n \in \Der_{F_0} (R)$ and $a_1,\ldots, a_n\in R$ such that $D_i a_j=1$ if $i=j$ and 
$0$ otherwise;
\item there exist $D_1,\ldots, D_n \in \Der_{F_0} (R)$ and $a_1\ldots, a_n\in R$ such that \\$\det(D_i a_j) \not\in m$;
\item $\Der_{F_0} (R)$ is a free module of rank $n$;
\item $\Rank(\Der_{F_0}(R))=n$.
\end{itemize}
\end{Teo}

\begin{Hyp}\label{H}
From now on, we will consider a commutative Noetherian regular ring $R$ with unity that contains a field, $F$, of characteristic zero 
satisfying: 
\begin{itemize}
\item[\rm{(1)}] $R$ is equidimensional of dimension $n$;
\item[\rm{(2)}] every residual field with respect to a maximal ideal is an algebraic extension of $F$;
\item[\rm{(3)}] $\Der_F (R)$ is a finitely generated projective $R$-module of rank $n$ such that $R_m\otimes_R \Der_F (R)=\Der_F(R_m)$.
\end{itemize}
\end{Hyp}
This hypothesis is inspired by the properties (i), (ii) and (iii) ($1.1.2$) in \cite{MeNa}. There, $R$ is a commutative Noetherian 
regular ring that contains a field, $F$, of characteristic zero satisfying $(1)$, $(2)$, but instead of ($3$) in Hypothesis \ref{H},
there exist $F$-linear derivations $\partial_1,\ldots, \partial_n \in \Der_{F_0} (R)$ and $a_1\ldots, a_n\in R$ such that 
$\partial_i a_j=1$ if $i=j$ and 
$0$ otherwise. In our hypothesis, part ($3$) includes more rings; 
for instance,  Remark \ref{Projective} gives an example of a ring that satisfies Hypothesis \ref{H} but not ($1.1.2$) in \cite{MeNa}. 
However, when $R$ a local ring the  properties  are the same by Theorem \ref{99}.
\begin{Rem}
Every regular finitely generated algebra over the complex numbers, $R$, satisfies Hypothesis \ref{H}. This is because, by Theorem $8.8$ \cite{Hartshorne}, $\Der_\CC(R)=\Hom_R (\Omega_{R/\CC}, R)$
and $\Omega_{R/\CC}$ is a
projective module such that $\Rank(\Omega_{R_m/\CC})=\Dim(R)$  for every maximal ideal $m\subset R$. 
\end{Rem}

\begin{Prop}\label{Ex1}
Let $R$ be a commutative Noetherian 
regular ring that contains a field, $F$, of characteristic zero satisfying $(1)$, $(2)$, and such that
there exist $F$-linear derivations $\partial_1,\ldots, \partial_n \in \Der_{F_0} (R)$ and $a_1\ldots, a_n\in R$ such that $\partial_i a_j=1$ if $i=j$ and 
$0$ otherwise.  Then, $R$ satisfies Hypothesis \ref{H}.
\end{Prop}
\begin{proof}
This follows from Theorem \ref{99}.
\end{proof}
A proof of Proposition \ref{Ex1}, along with several consequences, is contained in Remark $2.2.5$ in \cite{MeNa}.
\begin{Teo}\label{Ex2}
Let  $S$ be a commutative  Noetherian regular
ring that contains a field, $F$, of characteristic zero satisfying Hypothesis \ref{H}. If there is an element $f\in S$
such that $R=S/fS$ is a regular ring, then $R$ satisfies Hypothesis \ref{H}.
\end{Teo}
\begin{proof} 
We have that property ($1$) holds because $\dim S-1=\dim S_\eta-1=\dim R_m$ for every maximal ideal $m=\eta R\subset R$, where
$\eta\subset S$ is a maximal ideal of $S$ containing $fS$.
In addition, property ($2$) holds because every residual field of $R$ is a residual field of $S$.

We only need to prove property ($3$). Let $n=\dim(S)$. 
For every maximal ideal $\eta\subset S$ containing $fS$, we may pick a regular system of 
parameters, $y_1,\ldots,y_{n}$ for $S_\eta$ such that $y_1=f$. Then, by Theorem
 \ref{99}, there exist $\delta_i\in\Der_F (S_\eta)$
such that $\delta_i (y_j)=1$ if $i=j$ and zero otherwise; moreover, $\Der_F (S_\eta)$ is a free $S_\eta$-module of rank $n$ generated by
$\delta_1,\dots \delta_n$. 

Let $\varphi_f:\Der_F (S)\to R$ be the morphism defined by $\partial \to [\partial (f)]$, where $[\partial (f)]$ represents
 the class of  $\partial (f)$ in $R$.
Then, $S_\eta \otimes_S \Ker(\varphi_f)$ is isomorphic to 
$\{\delta \in \Der_F (S_\eta) : \delta(f)\in f\cdot S_\eta \}=S_\eta f\delta_1 \oplus S_\eta\delta_2\oplus\ldots\oplus S_\eta \delta_{n}$

Noticing that $f\cdot \Der_F (S)\subset \Ker\varphi_f$,
we define $N=\Ker\varphi_f/ (f\cdot\Der_F(S))$ and point out that it is a finitely generated $R$-module. Let $m=\eta R$. Then, 
$R_m \otimes_R N=R_m \delta_2\oplus\ldots\oplus R_m \delta_{n}=\Der_F (R_m)$, where the last equality uses Theorem \ref{99}.

We have a morphism $\psi:N\to \Der_F (R)$ defined by taking $\psi[\partial](r)=[\partial (r)]$, 
which is well defined by the definition of $N$. For every maximal ideal $m\subset R$, there is a natural morphism
$i_m:R_m\otimes_R \Der_F(R)\to \Der_F (R_m)$. We notice $( i_m\circ 1_{R_m}\otimes\psi )$ is an isomorphism between 
$R_m \otimes_R N$ and $\Der_F (R_m)$ for all maximal $m\subset R$. Therefore, 
$N_m\cong R_m\otimes \Der_F(R)\cong \Der_F (R_m)$ for all maximal $m\subset R$. Hence, $\psi$ is an isomorphism.
\end{proof}
\begin{Rem}\label{Projective}
It is worth pointing out that there are examples were $R$ satisfies Hypothesis \ref{H} but $\Der_{F} (R)$ is not free. 
Let $S=\RR[x,y,z]$ be the polynomial ring in three variables and coefficients over $\RR$. Let $f=x^2+y^2+z^2-1\in S$. Then, 
$R=S/fS$, the coordinate ring associated to the sphere, satisfies Hypothesis \ref{H} but $\Der_{\RR} (R)$ is not free. 
Let $\phi: R^3\to R$ be the morphism given by 
$(a,b,c)\to (ax,by,cz)$. 
Thus, $\Der_{\RR} (R)=Ker(\phi)$ by the proof of Theorem  \ref{Ex2}. Therefore, $\Der_{\RR} (R)$ is the projective module corresponding to  the tangent bundle of the sphere, and so it is not free. This example also shows that the conclusion of Theorem \ref{Ex2} does not hold for properties  (i), (ii) and (iii) ($1.1.2$) in \cite{MeNa}. In that sense, Hypothesis \ref{H} behaves better under regular subvarieties.

\end{Rem}
\begin{MainEx}\label{MainEx}
Let $(V,\pi V,K)$ be a DVR of mixed characteristic $p>0$, and let $F$ denote its fraction field. Let 
$S=V[[x_1,\ldots,x_{n+1}]]\otimes_V F$ be the tensor product of the power series ring with coefficients in $V$ and $F$. Let
$R=S/(f)S$
be a regular ring where $f=\pi-h$ for an element $h$ in the square of maximal ideal of $V[[x_1,\ldots,x_{n+1}]]$.
Then, $R$ satisfies Hypothesis \ref{H}.
\end{MainEx}
\begin{proof}
Since $S$ is as in Proposition  \ref{Ex1}  \cite{LyuUMC} pages $5880-5881$) and $\pi-h\in S$ is a regular element, we have that
$R$ satisfies Hypothesis \ref{H} by Theorem \ref{Ex2}.
\end{proof}

Let $F$ be a field of characteristic $0$ and $R$ a commutative Noetherian ring with unity containing $F$. We denote by $D(R,F)$ the ring of $F$-linear differential operators of $R$. This is a subring of
$\Hom_F (R,R)$ defined inductively as follows. The differential operators of order zero are 
the morphisms induced by multiplying by  elements in $R$.
An element $\theta \in \Hom_F(R,R)$ is a differential operator of order less than or equal to $j+1$ if $[\theta,r]:=\theta\cdot r -r\cdot\theta$
is a differential operator of order less than or equal to $j$. We have an induced filtration 
$\Gamma=(\Gamma^j)$ on $D(R,F)$ given by $\Gamma^j=\{\theta\in D(R,F) \mid \Ord(\theta)\leq j\}$. As a consequence of the definition, we
 have that $\Gamma_j\Gamma_i\subset \Gamma_{j+i}$ and that $\gr^\Gamma (D(R,F))=\oplus^\infty_{j=0} \Gamma^j / \Gamma^{j-1}$ is a commutative 
ring.

An example is given by a commutative Noetherian regular ring $R$ with unity that contains a field, $F$, of characteristic $0$, 
as in Proposition \ref{Ex1}. In this case, $D(R,F)=R[\partial_1,\ldots,\partial_n]\subset \Hom_F(R,R)$; moreover,
 $\gr^\Gamma(D(R,F))=R[y_1,\ldots,y_n]$ the polynomial ring with coefficients on $R$ and variables $y_1,\ldots y_n$ and
$\wgd(D(R,F))=\dim(R)$ (cf. Main Theorem in \cite{Bj2}, ($1.1.3$) and Theorem $1.1.4$ in \cite{MeNa}, and Theorem $2.17$ in \cite{Na}).
We would like to have similar properties for $D(R,F)$ and $\gr^\Gamma (D(R,F))$ when $R$ satisfies Hypothesis \ref{H}. 

We will denote by 
$D$ the subalgebra of $\Hom_F (R,R)$ generated by $R$ and $\Der_{F} (R)$, where $R=\Hom_R(R,R)\subset \Hom_F (R,R)$. 
We define an ascending filtration $\Gamma'_j$ of $R$-modules in $D$ inductively as follows. 
$\Gamma'_0=R$. Given $\Gamma'_j$, we take 
$\Gamma'_{j+1}$ as the  Abelian additive group generated by $\{\Gamma'_j,\Der_{F} (R) \cdot \Gamma'_j\}$. 
Since $\Gamma'_j$ is generated by multiplying
derivations, we have that for every $\delta\in\Gamma'_i$ and $f\in R$, $[\delta, f]=f\delta-\delta f\in \Gamma'_{j-1}$.
Therefore,
$\Gamma'_j$ is an $R$-submodule
of $D$ with respect to the structures induced by multiplication by the left or by the right. Additionally, $D\subset D(R,F)$ and 
$\Gamma'_j\subset \Gamma_j$ because $\Der_{F} (R)\subset \Gamma_1$.\\

We have that for every $s\in R$, 
$\Ad_s :D(R,A)\to D(R,A)$, defined by $\Ad_s(\delta)=s\delta-\delta s$, is nilpotent.
Let $m\subset R$ be a maximal ideal and $S=R\setminus m$ be the induced multiplicative system. Then, $S$ is a multiplicative set 
satisfying the Ore condition on the left and on the right in $D(R,A)$ and, as a consequence, in $D$. 
Hence, $S^{-1} D(R,F)$ 
and $S^{-1} D$ exist as filtered rings. 

\begin{Prop}\label{Dequal}
With the same notation as above, $D(R,F)=D$ as filtered rings.
\end{Prop}
\begin{proof}
Let $m\subset R$ be a maximal ideal and $S=R\setminus m$ be the induced multiplicative system.
We have that $S^{-1}\Gamma_j=S^{-1}\Gamma'_j$ by condition ($3$) in Hypothesis \ref{H}. Therefore,
$S^{-1} D=R_m[\Omega_{R_m,F}]=D(R_m,F)=S^{-1}D(R,F)$ as filtered rings. 
\end{proof}

For simplicity, we will denote $D(R,F)$ by $D$ and $(R\setminus m)^{-1} D(R,F)$ by $D_m$ for a maximal ideal $m\subset R$. 
We note that the inclusion $D\to D_m$ induces an inclusion $ \gr^\Gamma (D)\to  
\gr^{\Gamma_m}(D_m)$ of rings, and
$\gr^{\Gamma_m}(D_m)=R_m\otimes_R \gr^\Gamma (D)$. If $M$ is a left or right finitely generated $D$-module with a good filtration $\Pi$,
then $D_m\otimes_{D} M$ or $M\otimes_{D} D_m$, respectively,  has a filtration given by $R_m\otimes_R\Pi$ and 
$\gr^{\Pi_m} (D_m\otimes_{D} M)=R_m\otimes_R \gr^\Pi (M)$. 

We have that
$D_m$ is a left and right flat module over $D$, and
$D_m\otimes_D D_m \cong R_m\otimes_{R} D \cong   D\otimes_{R} R_m  \cong D_m$. If $M$ is a left or right 
finitely generated $D$-module, there exist a canonical isomorphism
$\Ext^i_{D_m} ( M_m,D_m)\cong S^{-1}\Ext^i_{D} ( M,D) \cong R_m\otimes_R \Ext^i_{D} ( M,D)$ 
for every $i\in \NN$.

We have, by Theorem \ref{99}, that for every maximal ideal $m\subset R$ there exist elements $x_1,\ldots,x_d\in R_m$ and
$F$-linear derivations $\partial_1,\ldots,\partial_d \in \Der_F (R_m)$ such that 
$\partial_i(x_j)=1$ if $i=j$ and zero otherwise. Therefore, $\wgd (D_m)=n$ and  $gr^{\Gamma_m}(D_m)=R[y_1,\ldots,y_n]$ is 
the polynomial ring with $n$ variables
and coefficients in $R_m$  \cite{Bj2}.

We recall the definition of a ring of differentiable type (cf. ($1.1$) in \cite{MeNa}).
\begin{Def}
A filtered ring $A$ is a \emph{ring of differentiable type} if its associated graded ring is commutative
Noetherian regular with unity and pure graded dimension.
\end{Def}
\begin{Teo}\label{GradedRing}
$(D,\Gamma)$ is a ring of differentiable type such that $\gr^\Gamma(D)$ is a ring of pure graded dimension $2n$.
\end{Teo}
\begin{proof} Let $\gr^\Gamma (D)$ be the associated graded ring. We will prove the proposition by parts.\\
\emph{ $\gr^\Gamma(D)$ is commutative:} This follows from the definition of the filtration $\Gamma$ on $D=D(R,F)$.\\
\emph{ $\gr^\Gamma(D)$ is Noetherian:} Let $\partial_1,\ldots,\partial_m$ be a set of generators for $\Der_{F} (R)$. 
Let $\phi : R[z_1,\ldots,z_m]\to \gr^\Gamma(D)$ be the morphism of commutative $R$-algebras defined by $z_i\to [\partial_i ]$. 
We have, by the definition of $\Gamma'=\Gamma$, that $\phi$ is surjectve. Hence $\gr^\Gamma(D)$ is Noetherian.\\
\emph{ $\gr^\Gamma(D)$ is regular:} Let $Q \subset \gr^\Gamma(D)$ be a prime ideal and $m\subset R$ be a maximal ideal that contains $Q\cap R$. Then
$\gr^\Gamma(D)_Q=(\gr^\Gamma(D)_m)_Q$ which is regular because $\gr^\Gamma(D)_m$ is a polynomial ring over $R_m$. \\
\emph{$\gr^\Gamma(D)$ has pure graded dimension $2n$:} Let $\eta$ be a maximal homogeneous ideal of $\gr^\Gamma(D)$. We claim that $m=\eta\cap R$ is
a maximal ideal of $R$. If not, there exist a maximal ideal $m'\subset R$ strictly containing  $m$. Then, $m'+\eta$ would be a proper ideal
of $\gr^\Gamma(D)$ that strictly contains  $\eta$. Hence,  $\gr^\Gamma(D)_\eta$ is the localization of $\gr^\Gamma(D)_m$ at a  maximal homogeneous ideal, then,
$\Dim(\gr^\Gamma(D)_\eta)=2n$ because $\gr^\Gamma(D_m)$ is a ring of pure graded dimension $2n$.
\end{proof}
\begin{Rem}\label{GradedSym}

Narv\'aez-Macarro \cite{Na} showed that if $S$ is a ring containing a field, $F$, of characteristic $0$ and  $\Der_{F}(S)$ is a projective $S$-modules of finite rank, then $\gr(D(S,F))\cong \Sym(\Der_F (S))$. 
Hence, we have that $\gr(D)\cong\Sym(\Der_F(R))$ by Hypothesis \ref{H}. 
\end{Rem}
\begin{Cor}
$D$ is left and right Noetherian.
\end{Cor}
\begin{proof}
This follows from Proposition $6.1$ in \cite{Bjork}.
\end{proof}

\begin{Prop}\label{DimDif}
$\wgd(D)= \Dim (R)$
\end{Prop}
\begin{proof}
Since $D$ is  left and right Noetherian, $\wgd(D)=\lpd(D)=\rpd(D)$ by Theorem $8.27$ in \cite{Rot}.  The value to this dimension is 
equal to the maximum integer $j$ such that $\Ext^j_D(M,R)\neq 0$ for some finitely generated 
$D$-module $M$ because $D$ is of differentiable type.  
As $R_m\otimes_R \Ext^j_{D} (M,D)=\Ext^n_{D_m} (M_m,D_m)= 0$ for every maximal ideal $m\subset R$ and integer $j>n$, we have that 
$\Ext^j_{D} (M,D)=0$ for every $D$-module $M$ and for $j>n$. Hence, $\wgd(D)\leq n$.
Likewise, $$R_m\otimes_R \Ext^n_{D} (R,D)=\Ext^n_{D_m} (R_m,D_m)\neq 0$$ for any $m\subset R$, so,
$\Ext^n_{D} (R,D)\neq 0$. Hence $\wgd(D)\geq n$.
\end{proof}

\section{The theory of the Bernstein-Sato polynomial and the Bernstein class of $D$}
Throughout this section we are adapting the results of Mebkhout and Narv\'aez-Macarro  to $R$ and $D$ \cite{MeNa}.
In particular, we show that the existence of the Bernstein-Sato polynomial and that 
the Bernstein class of $D$ is closed under localization at one element.

\begin{Def} 
Let $A$ be a ring of differentiable type.
Let $M\neq 0$ be a finitely generated left or right $A$-module. We define
$$
\Grade_A (M)=\Inf\{j: \Ext^j_{A} (M,A)\neq 0\}.
$$ 
\end{Def}

\begin{Prop}\label{122}  Let $A$ be a ring of differentiable type.
Let $M\neq 0$ be a finitely generated left or right $A$-module. Then,
$$\Dim(M)+\Grade_A (M)=\Dim(\gr^\Gamma(A)).$$ In particular, $\Dim(M)\geq \Dim(\gr(A))-\wgd(A)$. Moreover, we have that
$$\Codim_A (\Ext^i_A(M,A))\geq i$$ for all $i\geq 0$ such that $\Codim_A (\Ext^i_A(M,A))\neq 0$.
\end{Prop}
\begin{proof}
This is a generalized form of Theorem $7.1$ of section $2$ in \cite{Bjork}
given by Gabber \cite{Ga}. The proposition is stated in this form  in 
Mebkhout and Narv\'aez-Macarro's article
as Theorem $1.2.2$  \cite{MeNa}.
\end{proof}

\begin{Def}
Let $A$ be a ring of differentiable type.
Let $M$ be a finitely generated left or right  $A$-module. We say that $M$ is in the left or right \emph{Bernstein class}
if it has minimal dimension, i.e. $\Dim(M)=\Dim(\gr(A))-\wgd(A)$.
\end{Def} 

This class is closed under submodules, quotients and extensions.
The functor 
that sends
$M$ to $\Ext^n_{A}(M,A)$
is an exact contravariant functor that interchanges the left Bernstein
class and the right Bernstein class. Moreover, $M= \Ext^n_{A}( \Ext^n_{A}(M,A),A)$ 
naturally if $M$ is in either the left or the right Bernstein class, so that we have an anti-equivalence
of categories. 
In consequence, the modules in the Bernstein class have finite length as $A$-modules because it is a left and right Noetherian ring \cite[Prop. $1.2.7$]{MeNa}.
\begin{Prop}[Prop. in $1.2.7$ \cite{MeNa}]\label{127}
Let $A$ be a ring of differentiable type and let $f$ be an element in $A_0$. Let $M$ be an $A_f$-module finitely
generated, such that $\Ext^i_{A_f}(M,A_f)=0$ if $i\neq \wgd(A)$. Then, there exists a submodule $M'\subset M$ over $A$ such that $M'$
is finitely generated with minimal dimension and $M'_f=M$.
\end{Prop}

Through the rest of this section, $F(s)$ denotes the fraction field of the polynomial ring $F[s]$ over the field $F,$ and $D(s)$ denotes the ring $F(s)\otimes_F D$ with the filtration given by $F(s)\otimes_F\Gamma^i$. 
By $R(s)$, we mean the $F(s)$-algebra $F(s)\otimes_F R$.
Similarly, $D[s]$ denotes $F[s]\otimes_F D$ and $R[s]$ denotes $F[s]\otimes_F R$.

\begin{Prop}\label{211}
$R(s)$ is an $F(s)$-algebra equidimensional of dimension $\Dim(R)$.
\end{Prop}
\begin{proof}
This is an immediate consequence of Theorem $2.1.1$ in \cite{MeNa}.
\end{proof}

\begin{Prop}\label{214}
$D(s)$ is a ring of differentiable type with the filtration $F(s)\otimes_F\Gamma$ such that 
$\gr^{F(s)\otimes_F\Gamma} (D(s))$ is a ring of pure graded dimension $2\dim(R)$.
\end{Prop}
\begin{proof}
Since $D$ is a ring of differentiable type, 
$$\gr^{F(s)\otimes_F\Gamma} (D(s))=F(s)\otimes_{F[s]} F[s]\otimes_F\gr^\Gamma(D)=F(s)\otimes_{F[s]}\otimes_F\gr^\Gamma(D)[s]$$
is commutative, Noetherian and regular.
For the sake of simplicity, we will omit the filtration.
We claim that $\gr (D(s))$ has pure graded dimension $2\Dim(R)=2n$.
Let $\eta\subset \gr(D(s))$ be a maximal homogeneous ideal and $P=\eta\cap R$. 
Let $m\subset R$ be a maximal ideal containing $P$. We have that the ideal $\eta_m$, induced by $\eta$, 
is a maximal homogeneous ideal of 
$$(R\setminus m)^{-1} \gr(D(s))=F(s)\otimes_F \gr(D_m)=(F(s)\otimes_F R_m)[y_1,\ldots,y_n],$$
the polynomial ring with coefficients on $F(s)\otimes_F R_m$ and variables $y_1,\ldots,y_n$.
Then, $\Ht (\eta)=\Ht (\eta_m)=2n$  because $F(s)\otimes_F R_m$ is equidimensional of dimension $n$ by Theorem $2.1.4$ in \cite{MeNa}.
\end{proof}

Let $M$ be a left $D(s)$-module in the Bernstein class of $D(s)$. Let $N$ be a $D$-module in the Bernstein class of $D$
such that $F(s)\otimes_F N=M$. For every $\ell\in F$, the $D$-module $M_{\ell} :=N/(s-\ell)N$ is the Bernstein class of $D$.
\begin{Prop} \label{221}
With the same notaton as above, we have that
$$
\Dim_{D(s)}(M)\geq \Dim_D (N_\ell),
$$
for all but finitely many $\ell\in F$.
\end{Prop}
\begin{proof}
This is analogous to  the proof of Theorem $2.2.1$ in \cite{MeNa}.
\end{proof}

\begin{Prop}\label{223}
$\wgd(D(s))=\dim(R)=n$.
\end{Prop}
\begin{proof}
This is analogous to the proof of Theorem $2.2.3$ in \cite{MeNa}.
\end{proof}

Let $N[s]$ be the  free $R_f[s]$-module generated by a symbol ${\bf f^s}$ and let $N(s)=F(s)\otimes_F N[s]$. We give to $N[s]$ (resp. $N(s)$) a
structure of a left $D_f[s]$-module (resp. $D_f(s)$-module) as follows,
$$
\partial g{\bf f^s}=(\partial g+sgf^{-1}){\bf f^s}
$$
for every $\partial\in\Der_{F} (R)$ and every $g\in R_f[s]$ (resp. $g\in R_f(s)$ ). If M is a left $D$-module, we define
$
M_f[s]{\bf f^s}:=M_f[s]\otimes_{R_f[s]} N[s]=M\otimes_R N[s] $ and $ M_f(s){\bf f^s}:= N(s) \otimes_{R_f[s]} M_f(s)=M\otimes_R N(s).
$
This is a left $D_f[s]$-module ($D_f(s)$-module), and clearly, $M_f[s]\bf{f^s}$ ($M_f(s){\bf f^s}$) is
a finitely generated $D_f[s]$-module ($D_f(s)$-module) if $M$ is.
\begin{Prop}\label{311}
Let $M$ be a  left $D$-module in the Bernstein class and let $u\in M$.
Then, there exists a nonzero polynomial $b(s)\in F[s]$ and an operator $P(s)\in D[s]$ that satisfies the equation
$$
b(s)(u\otimes {\bf f^s})=P(s)f(u\otimes f^s)
$$in $M[s]{\bf f^s}$.
\end{Prop}
\begin{proof}
This is analogous to $3.1.1$ in \cite{MeNa}.
\end{proof}
\begin{Cor}\label{312}
If $M$ is a left $D$-module in the Bernstein class, the $M_f$ is a finitely generated $D$-module.
\end{Cor}
\begin{proof}
For $\ell \in \ZZ$, we  define a morphism of specialization $\phi_\ell : M_f[s]{\bf f^s}\to M_f$ by 
$\phi_\ell (us^i\otimes {\bf f^s})=\ell^i f^\ell u$, such that $\phi_\ell (P(s)v)= P(\ell)\phi_\ell (v)$.
Then, by applying this morphism to the result of Proposition \ref{311}, we have$$
b(\ell)f^\ell u=P(\ell)f^{\ell} u
$$and the conclusion follows.
\end{proof}

\begin{Cor}\label{PropLBC}
Let $M$ be a left $D$-module in the Bernstein class. Then, $M_f$ is also in the Bernstein class
for all $f\in R$.
\end{Cor}
\begin{proof}
Since $M_f$ is finitely generated as a $D$-module, it suffices to show that $\Dim_{\gr(D)}(\gr(M_f))$ is $n$.
Since $R_m\otimes_R M$ is in the Bernstein class of $D_m$, we have that $M_f$ is in the Bernstein class of $D_m$ for every maximal ideal
$m\subset R$ by
Theorem $2.2.3$ in \cite{MeNa}. Thus, $\Dim_{\gr(D_m)}(\gr((M_m)_f))=n$ and, therefore 
$\Dim_{\gr(D)}(\gr(M_f))=n$.
\end{proof}

\begin{proof}[Proof of Theorem \ref{MainOne}]
This is a consequence of Theorem \ref{GradedRing}, Proposition \ref{DimDif} and Corollary \ref{PropLBC}.
\end{proof}

\section{Local cohomology}
Let us recall the family of functors introduced by Lyubeznik \cite{LyuDMod}.
If $Z \subset \Spec ( R)$ is a closed subset and $M$ is an $R$-module, we denote by 
$H^i_Z (M)$
the $i$-th local cohomology module of $M$ with support in $Z$. This can be calculated via the $\check{\mbox{C}}$ech complex 
as follows:
\begin{equation}\label{LC1}
0\to M\to \oplus_{i} M_{f_i}\to\ldots  \to \oplus_i M_{f_1\cdots\hat{f}_1 \cdots f_\ell}\to
M_{f_1\cdots f_\ell} \to 0\end{equation}
where $Z=\cV(f_1,\ldots,f_\ell)=\{P\in\Spec(R): (f_1,\ldots,f_\ell)\subset P\}$
 
For two
closed subsets of $\Spec(R)$, $Z_1\subset Z_2$, there is a long exact sequence of functors.  In particular, 
$H^i_Z (M)=H^i_I (M)$.

\begin{equation}\label{LC2}
\ldots\to H^i_{Z_1}\to H^i_{Z_2}\to H^i_{Z_1/Z_2}\to \ldots
\end{equation}
\begin{Def}\label{Notation} We say that $\cT$ is a \emph{Lyubeznik functor} if has the form $\cT =\cT_1\circ \dots\circ\cT_t$, where every functor $\cT_j$
is either $H_{Z_1},$ $H^i_{Z_1\setminus Z_2},$ or the kernel, image or cokernel of some arrow in the previous long exact sequence
for closed subsets $Z_1,Z_2$ of $\Spec(R)$ such that $Z_2\subset Z_1$.
\end{Def}
\begin{Lemma}\label{LAP1}
Let $M$ be a left $D$-module in the Bernstein class. Then, $\cT(M)$ has a natural structure of $D$-module such that 
it belongs to the Bernstein class for every functor $\cT$ as in Definition \ref{Notation}.
\end{Lemma}
\begin{proof}
$M_f$ has a structure of $D$-module given by 
$$\partial\cdot m/f^\ell=(f^\ell\delta\cdot m -\delta(f^\ell)m)/f^{2\ell}$$ for every $\delta\in\Der_F(R)$. Then, $\cT(M)$ is a $D$-modules by Examples $2.1$ in \cite{LyuDMod}.
Since M is in the Bernstein class,
$M_f$ is in the Bernstein class and $M\to M_f$ is a morphism of $D$-modules
by Corollary \ref{PropLBC}.  Since the 
Bernstein class is closed under extension, submodules and quotients, 
every element in the complexes (\ref{LC1}) and (\ref{LC2}) as well as the 
kernels, images and homology groups are in the same class, and the result follows.
\end{proof}
\begin{Lemma}\label{LAP2}
Let $M$ be a left $D$-module in the Bernstein class. Then, $\Ass_R (M)$ is finite.
\end{Lemma}
\begin{proof}
Suppose $M\neq 0$.
Let $M_1=M$ and  $P_1$ be a maximal element in the set of the associated primes of $M_1$.
Then, $N_1=H^0_{P_1}(M_1)$  a nonzero $D$-submodule of $M_1$, and it has only one associated
prime. Given $N_{j}$ and $M_j$, set $M_{j+1}=M_{j}/N_j$.  If $M_{j+1}\neq 0$, let $P_{j+1}$ be a maximal element in the set of the
associated primes of $M_{j+1}$. Then 
$N_{j+1}=H^0_{P_j} (M_{j+1})$ has only one associated
prime. If $M_{j+1}=0$, set $N_{j+1}=0$. Since $M_1=M$ has finite length as a $D(R,A)$-module, there exist $\ell\in \NN$ such that $M_j=0$ for $j\geq \ell$, and then
$\Ass(M)\subset\{P_1,\dots,P_\ell\}$.
\end{proof}
\begin{Teo}\label{CorAss}
Let $R$ be a ring that satisfies Hypothesis \ref{H} and let $M$ be an $D$-module in its left Bernstein class. 
Then, $\Ass_R (\cT(M))$ is finite for every functor $\cT(-)$ as in Definition \ref{Notation}. In particular, this holds for $H^i_I(R)$ for every $i\in\NN$ and ideal $I\subset R$.
\end{Teo}
\begin{proof}
This follows immediately form Lemmas \ref{LAP1} and \ref{LAP2}. 
\end{proof}
\begin{Cor}\label{CorLCMC}
Let $(R,m,K)$ be a regular local ring of mixed characteristic $p>0$. Then, the set of associated primes of $\cT (R)$ that does not
contain $p$
is finite for every functor $\cT$ as in Definition \ref{Notation}.
\end{Cor}
\begin{proof}
Let $\hat{R}$ be the completion of $R$ with respect to the maximal ideal. 
Then, the set of associated primes of $\cT (R)$ that does not contain $p$ 
is finite if the set of associated primes of $\cT(\hat{R})=\hat{R}\otimes_R \cT(R)$ that does not contain $p$ 
is finite. We can assume without loss of generality that $R$ is complete.
Thus, $R=V[[x_1,\ldots,x_{n+1}]]/(p-h)V[[x_1,\ldots,x_{n+1}]]$ where $(V,p V,K)$ is a DVR of unramified mixed characteristic $p>0$ and $h$ is an element in the 
square of maximal ideal of $V[[x_1,\ldots,x_{n+1}]]$ by Cohen Structure Theorems. Let $F$ be the fraction field of $V$. It suffices to show that 
$\Ass_R (F\otimes_V \cT (R))=\Ass_R (\cT (F\otimes_V R)$ is finite, which follows from our main example  and Theorem \ref{CorAss}.
\end{proof}
\begin{proof}[Proof of Theorem \ref{MainTwo}]
This is an immediate consequence of Corollary \ref{CorLCMC}. 
\end{proof}
\section*{Acknowledgments}
I would like to thank Mel Hochster for his valuable comments and helpful discussions.  
I also wish to thank Luis Narv\'aez-Macarro for reading this
manuscript and for his helpful comments and suggestions.  
Thanks are also due to the National Council of Science and Technology of Mexico by its support through grant $210916.$

{\sc Department of Mathematics, University of Michigan, Ann Arbor, MI $48109$--$1043$, USA.}\\
{\it Email address:}  \texttt{luisnub@umich.edu}

\begin{thebibliography}{ZZZZZ}
\bibitem[Bj1]{Bjork}
J.E. Bjork, Rings of differential operators, North Holland, Amsterdam, $1979$.

\bibitem[Bj2]{Bj2}
J.E. Bjork, The global homological dimension of some algebras of differential operators. Invent. Math. Vol. $17$, $1992$.

\bibitem[Ga]{Ga}
O. Gabber, Equidimensionalit$\acute{\hbox{e}}$ de la vari$\acute{\hbox{e}}$t$\acute{\hbox{e}}$ caract$\acute{\hbox{e}}$ristique (expos$\acute{\hbox{e}}$ de
O. Gabber r$\acute{\hbox{e}}$dig$\acute{\hbox{e}}$ par T. Levasseur), unpublished manuscript.

\bibitem[Ha]{Hartshorne}
Hartshorne, Robin Algebraic geometry. Graduate Texts in Mathematics, No. $52.$ Springer-Verlag, New York-Heidelberg, $1977$.

\bibitem[Ly1]{LyuDMod}
G. Lyubeznik, Finiteness properties of local cohomology modules (an application of $D$-modules to commutative algebra), Invent. Math. $113$ ($1993$), no. 1, $41$-$55$. 


\bibitem[Ly2]{LyuUMC}
G. Lyubeznik, Finiteness properties of local cohomology modules
for regular local rings of mixed characteristic: the unramified case,
Special issue in honor of Robin Hartshorne. Comm. Algebra $28$ ($2000$), no. $12$, $5867$-$5882$.

\bibitem[Ma]{Matsumura}
H. Matsumura, Commutative algebra, Second edition, Mathematics Lecture Note Series, 56. Benjamin/Cummings Publishing Co., Inc., Reading, Mass., $1980$.

\bibitem[MeNa]{MeNa}
Z. Mebkhout and L. Narv\'aez-Macarro, La th$\acute{\hbox{e}}$orie du polyn$\hat{\hbox{o}}$me de Bernstein-Sato pour les alg$\grave{\hbox{e}}$bres de Tate et de Dwork-Monsky-Washnitzer. Ann. Sci. $\acute{\hbox{E}}$cole Norm. Sup. ($4$) $24$ ($1991$), no. $2$, $227$–-$256$.

\bibitem[Na]{Na}
L. Narv\'aez-Macarro,
Hasse-Schmidt derivations, divided powers and differential smoothness,
Ann. Inst. Fourier (Grenoble) $59$ ($2009$), no. $7$, $2979$–-$3014$. 

\bibitem[Ro]{Rot}
J. J. Rotman, An introduction to homological algebra. 2nd ed. New York: Springer, ($2009$).

\end{thebibliography}
\end{document}